\documentclass{amsart}
\usepackage{amsxtra,amssymb}
\usepackage{color}
\usepackage{verbatim}

\newtheorem{Th}{Theorem}
\newtheorem{Pp}[Th]{Proposition}
\newtheorem{Le}[Th]{Lemma}
\newtheorem{Co}[Th]{Corollary}
\newtheorem{Cj}[Th]{Conjecture}
\newtheorem{Rm}[Th]{Remark}
\theoremstyle{remark}

\newcommand{\dom}{\trianglerighteq}
\newcommand{\sdom}{\triangleright}
\newcommand{\ndom}{\ntrianglerighteq}

\newcommand{\partof}{\vdash}
\newcommand{\innp}[2]{\langle#1,#2\rangle}
\newcommand{\wssp}{\mathcal{SSP}}
\newcommand{\wsp}{\mathcal{SP}}

\newcommand{\clr}{\textcolor{red}}
\newcommand{\clg}{\textcolor{green}}
\newcommand{\clb}{\textcolor{blue}}

\begin{document}

\title{The Schur Cone and the Cone of Log Concavity}

\author{Dennis E. White}
\address[Dennis E.~White]{School of
Mathematics, University of Minnesota\\
127 Vincent Hall, 206 Church St SE\\
Minneapolis, MN 55455--0488}
\email{white@math.umn.edu}
\urladdr{http://www.math.umn.edu/~white}

\date{\today}

\begin{abstract}
Let $\{h_1,h_2,\dots\}$ be a set of algebraically
independent variables.  We ask which vectors
are extreme in the cone generated by
$h_ih_j-h_{i+1}h_{j-1}$ ($i\geq j>0$)
and $h_i$ ($i>0$).  We call this cone
the \textit{cone of log concavity}.
More generally, we ask
which vectors are extreme in the
cone generated by Schur functions
of partitions with $k$ or fewer parts.
We give a conjecture characterizing which vectors are extreme
in the cone of log concavity.  We prove this characterization
in one direction and give partial results in the other direction.
\end{abstract}

\maketitle

\section{Introduction, partitions and symmetric functions}\label{S:parts}

Let $\{h_1, h_2, \dots\}$ be a set
of algebraically independent variables.
We ask which polynomials in these
variables can be written as positive sums of products of
polynomials of the form $h_ih_j-h_{i+1}h_{j-1}$ ($i\geq j>0$)
and $h_i$ ($i>0$).  
Such sums of products form a cone inside the algebra generated by these
variables.  We call this cone the \textit{cone of log concavity}.  It is natural
to ask which of the generating vectors of this cone are extreme and which
are not.  That is, which can be written as positive linear combinations of the others
and which are required to define the cone.

We can view the $h_i$ as the homogenous symmetric functions in a set of indeterminates
$\{x_1,x_2,\dots\}$.  In this setting, the monomials of products of the $h_i$ are
a basis for the vector space of symmetric functions, and the monomials of homogenous
degree $N$ are a basis for the vector space of symmetric functions of that degree.  
Since the variables $\{h_1,h_2,\dots\}$ are algebraically independent, we will
never have to refer to the underlying indeterminates $\{x_1,x_2,\dots\}$.

Having placed our problem in the context of symmetric functions, we need some 
preliminary definitions and results
concerning partitions and symmetric functions.
This material may be found in many other sources, most notably
in~\cite{Macd} and in~\cite{Stan2}.

If $\lambda=(\lambda_1,\lambda_2,\dots,\lambda_m)$ with integers
$\lambda_1\geq\lambda_2\geq\dots\geq\lambda_m>0$
and $N=\lambda_1+\lambda_2+\dots+\lambda_m$, then
$\lambda$ is called a \textit{partition of $N$}, and
we write $\lambda\partof N$ and $|\lambda|=N$.  The integers $\lambda_i$
are called the \textit{parts} of the partition and $m=l(\lambda)$
is the \textit{number of parts}.  Another common notation for
partitions uses an exponential form.  If the part $k$ appears
$t_k$ times in the partition, we write $k^{t_k}$.  Thus the partition
of $18$, $(4,4,2,2,2,1,1,1,1)$, can be written $4^22^31^4$.

Let $\mathcal P_N$ be the set of partitions of $N$, 
$\mathcal P_N^k$ be the set of partitions of $N$ with $k$ or fewer parts
and $\mathcal P^k$ be the set of partitions with $k$
or fewer parts.  Let $p(N)$ be the size of $\mathcal P_N$.

A partition $\lambda$ is sometimes called a \textit{shape}, especially
when it is described by a \textit{Ferrers diagram}, an array of 
left-justified cells with
$\lambda_1$ cells in the first row, $\lambda_2$ cells in the second
row, etc.

If $\lambda\partof N$ and $\mu\partof N$, we say $\lambda$ \textit{dominates}
$\mu$ if $\lambda_1+\dots+\lambda_i\geq\mu_1+\dots+\mu_i$ for
all $i$ and we write $\lambda\dom\mu$ (and $\lambda\sdom\mu$ if
$\lambda\dom\mu$ and $\lambda\neq\mu$).  In these partial sums, if one partition has
more parts than the other, we pad with
parts of size $0$ as necessary.  Dominance determines a
partial order on $\mathcal P_N$.

If positive integers are placed in
the cells of the shape $\lambda$, the resulting
figure is called a \textit{tableau}.  The \textit{content} of a tableau
is a vector $\rho=(\rho_1,\rho_2,\dots)$ where $\rho_i$ is the number
of $i$'s in the tableau.  Vectors such as $\rho$ are called
\textit{compositions}.

If the entries of the tableau weakly increase across rows and strictly increase
down columns, the tableau is called a \textit{semistandard Young tableau}, or
\textit{SSYT}.  The number
of SSYT of shape $\lambda$ and content $\rho$ is
$K_{\lambda,\rho}$, called the \textit{Kostka number}.  A well-known
property of SSYT is that $K_{\lambda,\rho}$ does
not depend upon the order of the entries in the vector $\rho$,
so $\rho$ is usually assumed to be a partition.  In the next section, we shall
describe a  bijection on SSYT which proves this
property.

We often encounter shapes more general than partitions.  Given partitions
$\lambda$ and $\mu$, we say
$\mu\leq\lambda$ if each $\mu_i\leq\lambda_i$. Write $\lambda/\mu$
to denote the diagram obtained by removing the cells of the Ferrers diagram
of $\mu$ from
the cells of the Ferrers diagram of $\lambda$. This diagram is called a \textit{skew shape}, and
the idea of a SSYT extends naturally to skew shapes.

If $T$ is a (possibly skew) SSYT, then $w(T)$, called the
\textit{word of $T$}, is the word obtained by reading the entries
in $T$ from right to left across the first (top) row, then right
to left across the second row, etc.  If $\alpha$ is
a subset of the letters appearing in $T$, then $w_{\alpha}(T)$
is the subword of $w(T)$ which uses just letters in $\alpha$.

A word, using the letters $t_1<t_2<\dots<t_p$,
is a \textit{lattice word} if, at any point in the word
(reading left to right),
the number of $t_i$'s which have appeared is $\geq$ the number
of $t_{i+1}$'s which have appeared.

As mentioned earlier, the $\{h_1,h_2,\dots\}$ described above are usually defined
to be the homogeneous symmetric functions in some set of indeterminates
$x_1,x_2,\dots$.  In this paper we will never need to refer
to this underlying variable set.  The fact that the $h$'s are
algebraically independent gives us the freedom to move around
among symmetric function bases without regard to the underlying
set of indeterminates.

We write $h_{\rho}=h_{\rho_1}h_{\rho_2}\dots$, where
$\rho\partof N$.  The $h_{\rho}$, $\rho\partof N$,
form a basis of a vector space $\Lambda^N$
of dimension $p(N)$.

We will use another basis, the Schur functions $s_{\lambda}$, extensively.
We connect this basis with the $h_{\rho}$ in two ways.  The first
is the equation
$$
h_{\rho}=\sum_{\lambda\partof N}K_{\lambda,\rho}s_{\lambda}\,,
$$
where $\rho\partof N$.
The second is the Jacobi-Trudi identity,
$$
s_{\lambda}=\det(h_{\lambda_i-i+j})_{1\leq i,j\leq n}
$$
where $n\geq l(\lambda)$.

Symmetric functions which can be written in the Schur function
basis with integer coefficients are called \textit{Schur-integral};
symmetric functions which can be written in the Schur function basis
with non-negative coefficients are called \textit{Schur-positive}.

\section{The Littlewood-Richardson rule}\label{S:LRRule}

If two Schur functions are multiplied together, the resulting symmetric
function can be expanded as a linear combination of Schur functions.
More generally, suppose $(\rho^1,\rho^2,\dots,\rho^k)$ is a vector
of partitions of $n_1,n_2,
\dots,n_k$ respectively and $N=n_1+n_2+\dots+n_k$.
Write
$$
\prod_{i=1}^k s_{\rho^i}=\sum_{\lambda\partof N} 
c_{\rho^1,\dots,\rho^k}^{\lambda} s_{\lambda}
$$
The coefficients
$c_{\rho^1,\dots,\rho^k}^{\lambda}$ are the well-known \textit{Littlewood-Richardson coefficients},
whose computation is described below. They are
non-negative integers, making the product of Schur functions both Schur-positive and
Schur-integral.

The Littlewood-Richard coefficients are computed as follows.
Let $\rho=\rho^1\vee\dots\vee\rho^k$ be the composition formed by
concatenating the parts of the partitions $\rho^1,\rho^2,\dots,\rho^k$.
For example, if $\rho^1=(2,2,1)$, $\rho^2=(4,2)$ and $\rho^3=(3,1)$, then
$\rho=(2,2,1,4,2,3,1)$.  We next form a SSYT, $T$, of shape $\lambda\partof N$ and content $\rho$.
We let $\alpha^i$ denote the subset of letters in $T$ corresponding
to the partition $\rho^i$.  In our example, $\alpha^1=\{\clr1,\clr2,\clr3\}$, $\alpha^2=\{\clb4,\clb5\}$
and $\alpha^3=\{\clg6,\clg7\}$.

Finally, we say $T$ \textit{is LW} if, for each $i$, $w_{\alpha^i}(T)$ is a lattice word.
For example, taking
$\rho$ as above and $\lambda=(5,4,4,2)$, this tableau $T$ is LW:

$$
T=\begin{array}{ccccc}
\clr1&\clr1&\clb4&\clb4&\clg6\\
\clr2&\clr2&\clb5&\clg6\\
\clr3&\clb4&\clg6&\clg7\\
\clb4&\clb5
\end{array}
$$

This is because the three words
$(\clr{1\,1\,2\,2\,3})$, $(\clb{4\,4\,5\,4\,5\,4})$ and $(\clg{6\,6\,7\,6})$ are each a lattice word.

\begin{Th}\label{Th:LRRule}
The coefficient $c_{\rho^1,\dots,\rho^k}^{\lambda}$ is  the number of SSYT
of shape $\lambda$, content $\rho$ which are LW.
\end{Th}

The Littlewood-Richardson rule has many proofs~\cite{Stan2}.  One (\cite{JaKe}) uses a well-known switching rule
which can also be used to prove the Kostka numbers are independent of the order of the content.
This rule, which is a rephrasing of the \textit{jeu de taquin} of Sch\"utzenberger~\cite{Stan2}, swaps a letter from
one alphabet through a tableau in a different alphabet. We make this more precise.

Suppose inside the (possibly skew) SSYT, $T$, the two letters $*$ and $0$ appear, with $0<*$ and no letter $x$
such that $0<x<*$.  (We say such $0$ and $*$ are \textit{contiguous} or \textit{appear contiguously}.)
We swap the order of $0$ and $*$ as follows.  Whenever $0$ and $*$ appear
in a column, we call them \textit{paired} and we swap the paired $0$ and $*$.  And in any row, we swap
the unpaired $0$'s with the unpaired $*$'s.  The resulting tableau, $r(T)$, will have the $0$'s and $*$'s occupying
the same set of cells, with multiplicities unchanged, but with $*<0$.

For example, if
$$
T=
\begin{array}{ccccccccccc}
&&&&&&&&0&0&0\\
&&&0&0&0&0&*&*&\\
0&*&*&*&*\\
*
\end{array}
$$
then
$$
r(T)=
\begin{array}{ccccccccccc}
&&&&&&&&*&0&0\\
&&&*&*&*&0&0&0&\\
*&*&*&0&0\\
0
\end{array}
$$

By iterating this process, two alphabets can be made to swap positions.  That is, if $T$ is a skew SSYT of shape $\lambda/\mu$
which uses two alphabets, $\alpha$ and $\beta$, with the $\alpha$
alphabet less than the $\beta$ alphabet (written $\alpha<\beta$), then
repeatedly passing letters from one alphabet through the other gives a second skew SSYT, $S$, of shape $\lambda/\mu$,
using the same alphabets, but with $\beta<\alpha$.

Furthermore, certain properties of these alphabets are maintained after this swapping.  Write $S=r_{\beta<\alpha}(T)$
and $T=r_{\alpha<\beta}(S)$ to represent this swapping, and let $T_{\alpha}$ denote the skew subtableau of $T$ which uses
only the alphabet $\alpha$.  We have the following theorem, which appears in~\cite{BSS}:
\begin{Th}\label{Th:swapprop}
If $S=r_{\beta<\alpha}(T)$, then $T_{\alpha}$ is LR if and only if $S_{\alpha}$ is LR and $T_{\beta}$ is LR if and only if
$S_{\beta}$ is LR.
\end{Th}

\section{The cone of log-concavity}\label{S:cone}

If $A$ is a multiset from $\mathcal P^k$, define
$$
wt(A)=\sum_{\lambda\in A}|\lambda|
$$
and
$$
s_A=\prod_{\lambda\in A}s_{\lambda}\,.
$$
The homogeneous degree of $s_A$ (as a polynomial in
the $h$'s) is $wt(A)$.  Define
$$
\wsp_N^k = \{A\mid wt(A)=N\}\,.
$$

The $(N, k)$-\textit{Schur cone} is
$$
\mathcal C_N^k=\bigl\{\sum_{A\in\wsp_N^k}
c_A s_A\mid c_A\geq0\bigr\}
$$

A function $s_A$, $A\in\wsp_N^k$ is \textit{extreme} in
$\mathcal C_N^k$
if it cannot be written as a positive linear combination of
other $s_B$, $B\in\wsp_N^k$.
We ask, for a given $k$, which elements $A\in\wsp_N^k$ yield
$s_A$ which are extreme in this cone.

We distinguish two obvious special cases.  When $k=1$, $s_A=h_{\lambda}$
where $\lambda$ is the partition whose parts are the 1-row partitions of
$A$.  Since the $h_{\lambda}$ form a basis of $\Lambda^N$
and are the only vectors
defining $\mathcal C_N^1$, they are the extreme vectors.

When $k\geq N$, then since the product of Schur functions is Schur-positive,
the Schur functions $s_{\lambda}$ are the extreme vectors.

It follows from the Jacobi-Trudi identity that the cone $\mathcal C_N^2$
consists of positive linear combinations of products of factors
of the form
$$
h_i h_j-h_{i+1}h_{j-1}\quad\text{and}\quad h_i\quad i\geq j\geq1\,.
$$

Thus, we call $\mathcal C_N^2$ the \textit{cone of log concavity}.

There are many elements $A\in\wsp_N^2$ which are not
extreme in $\mathcal C_N^2$.  For example,
$$
s_{(3,1)}s_{(2)} = s_{(3,2)}s_{(1)} + s_{(1,1)}s_{(4)}\,.
$$
In fact, the extreme set of $\mathcal C_6^2$ is just these $13$ elements:
\begingroup
\renewcommand*{\arraystretch}{1.4}
$$
\begin{array}{lll}
s_{(6)}&s_{(4)}s_{(1,1)}&s_{(3)}s_{(2,1)}\\
s_{(5,1)}&s_{(3,1)}s_{(1,1)}&{s_{(2,1)}}^2\\
s_{(4,2)}&s_{(2,2)}s_{(2)}&s_{(2)}{s_{(1,1)}}^2\\
s_{(3,3)}&s_{(2,2)}s_{(1,1)}&{s_{(1,1)}}^3\\
s_{(3,2)}s_{(1)}
\end{array}
$$
\endgroup
In this paper we conjecture a simple characterization of the extreme
elements of $\wsp_N^2$.
We give
a proof of this conjecture in one direction and we prove an important special case
in the other direction.

\section{The extreme set}\label{S:extreme}

The conjectured characterization of the extreme elements of $\mathcal C_N^2$
is the following.
\begin{Cj}\label{Cj:firstconj}
The collection of pairs $A\in \wsp_N^2$ is in the extreme set of
$\mathcal C_N^2$
if and only if no pair of partitions $\{\lambda,\mu\}$ in $A$ satisfies
any one of the following conditions:
\begin{enumerate}
\item
$\lambda=(\lambda_1\geq\lambda_2>0)$,
$\mu=(\mu_1\geq\mu_2>0)$, with
$$
\lambda_1>\mu_1\geq\lambda_2>\mu_2\,;
$$
\item
$\lambda=(\lambda_1>\lambda_2>0)$,
$\mu=(\mu_1>0)$, with
$$
\lambda_1\geq\mu_1\geq\lambda_2\,;
$$
\item
$\lambda=(\lambda_1>0)$,
$\mu=(\mu_1>0)$.
\end{enumerate}
\end{Cj}

If no pair of partitions in $A$ satisfies any of these conditions, we say $A$ is
\textit{nested}.  The proof of one direction is easy.

\begin{Th}\label{Th:conjonedir}
If $A$ is not nested then $A$ is not in the extreme set of
$\mathcal C_N^2$.
\end{Th}

\begin{proof}
Suppose a pair $\{\lambda,\mu\}$ satisfies the first condition.
This implies $\lambda_1\geq \mu_1+1$ and $\lambda_2-1\geq \mu_2$.
Therefore, by Jacobi-Trudi,
\begin{equation}\label{eq:syzygy1}
s_{\lambda}s_{\mu}=s_{(\lambda_1,\mu_2)}s_{(\mu_1,\lambda_2)}
+s_{(\lambda_1,\mu_1+1)}s_{(\lambda_2-1,\mu_2)}\,.
\end{equation}

Suppose a pair $\{\lambda,\mu\}$ satisfies the second condition.
If $\lambda_1>\mu_1$ then by Jacobi-Trudi,
\begin{equation}\label{eq:syzygy2}
s_{\lambda}s_{\mu}=s_{(\lambda_1)}s_{(\mu_1,\lambda_2)}+
s_{(\lambda_2-1)}s_{(\lambda_1,\mu_1+1)}\,.
\end{equation}
If $\mu_1>\lambda_2$, by Jacobi-Trudi
\begin{equation}\label{eq:syzygy3}
s_{\lambda}s_{\mu}=s_{(\lambda_2)}s_{(\lambda_1,\mu_1)}+
s_{(\lambda_1+1)}s_{(\mu_1-1,\lambda_2)}\,.
\end{equation}

Finally, suppose a pair $\{\lambda,\mu\}$ satisfies the third condition.
Then
\begin{equation}\label{eq:syzygy4}
s_{\lambda}s_{\mu}=s_{(\lambda_1,\mu_1)}+s_{(\lambda_1+1)}s_{(\mu_1-1)}\,.
\end{equation}
\end{proof}

Let $\wssp_N$ denote the nested sets $A\in \wsp_N^2$.
Thus, the extreme set of
$\mathcal C_N^2$
is contained in $\wssp_N$.

For $A\in \wsp_N^2$, let $\phi(A)$ be the partition defined by the parts
of the partitions in $A$.
For example, if $A=\{(4,2),(3,1),(3,2),(2)\}$, then $\phi(A)=43^22^31$.

Several $A\in \wsp_N^2$ might have the same $\phi(A)$.  For $\lambda\vdash
N$, let $\wsp_{\lambda}^2=\{A\in \wsp_N^2\mid\phi(A)=\lambda\}$.  For example,
if $\lambda=42^21$, then $\{(4,2),(2,1)\}$ and $\{(4,1),(2,2)\}$ are both
elements of $\wsp_{\lambda}$.
Similarly, define $\wssp_{\lambda}=\wsp_{\lambda}^2\cap \wssp_N$.  Note that
for every $\lambda$, $\wssp_{\lambda}\neq\emptyset$.

\begin{Rm}\label{Rm:evenodd}
For $A\in\wssp_{\lambda}$,
if $\lambda$ has an even number of parts, then all the partitions of $A$
have two parts, while if $\lambda$ has an odd number of parts, then
exactly one partition of $A$ will have one part (and the remaining
partitions in $A$ will have two parts).
\end{Rm}

Note that when $s_A$ is expanded in Schur functions,
its support lies above $\phi(A)$ in dominance order, and its coefficients
are non-negative.  This is a consequence of the 
Littlewood-Richardson rule.

\begin{Pp}\label{Pp:domprop}
If $A\in \wsp_N^2$, then
$$
s_A=\sum_{\mu\dom\phi(A)}c_A^{\mu}s_{\mu}\,,
$$
with $c_A^{\phi(A)}=1$ and $c_A^{\mu}\geq 0$.
\end{Pp}

Our primary tool in proving elements in $\wssp_N$ are extreme
in $\mathcal C_N^2$ is the well-known Farkas' Lemma (see~\cite{Schr}).
Farkas' Lemma states that a vector $\mathbf v$ is extreme in a cone
if and only if there is a separating hyperplane, i.e., a hyperplane $P$
such that $\mathbf v$ lies on one side of $P$ and all other generating
vectors lie on the other side of $P$.

Since we are working in $\Lambda^N$ and
using the Schur functions as our basis, it is natural to determine separating hyperplanes
by using the standard symmetric function inner product $\innp{\cdot}{\cdot}$
for which the Schur functions are orthonormal.

Suppose
$A\in \wssp_N$ and let $f$ be a symmetric function
such that $\innp{f}{s_B}\leq 0$ for $B\in \wssp_N$, $B\neq A$, and
$\innp{f}{s_A}>0$.  Then we
say $f$ \textit{separates} $A$. Restating Farkas' Lemma in our context:

\begin{Th}\label{Th:Farkas}
There is a symmetric function $f$ which separates $A$ for $A\in \wssp_N$
if and only if $A$ is extreme in $\mathcal C_N^2$.
\end{Th}

To prove Conjecture~\ref{Cj:firstconj}, we seek therefore
a set of separating functions, one for each $A$.
We now use Proposition~\ref{Pp:domprop} to reduce the amount
of work we must do in finding separating functions.  In effect,
Proposition~\ref{Pp:domprop} states that we need only work above
$\phi(A)$ in dominance order. To formalize this idea, we introduce this definition.
Let $\lambda=\phi(A)$, $A\in\wssp_N$.  
We say the symmetric function $f$ \textit{separates $A$ from above} if
\begin{enumerate}
\item
$f$ is Schur integral;
\item
$\innp{f}{s_A}>0$;
\item
$\innp{f}{s_B}\leq 0$ for all $B$ such that
$\phi(B)\dom\lambda$, $B\neq A$.
\end{enumerate}

For example, for $N=6$ and $A=\{(2,1),(2,1)\}$, we have $\lambda=2^21^2$.
If we take $f = s_{2^21^2}+s_{2^3}+s_{31^3}-s_{321}$,
then $f$ separates $A$ from above.

\begin{Le}\label{Le:separate}
If $f$ separates $A$ from above,
then there is a symmetric function $g$ which
separates $A$.
\end{Le}

\begin{proof}
Let $I$ be a dual order ideal in the dominance poset (see~\cite{Stan1}
for definitions) with $\lambda = \phi(A)\in I$.  Let $u$
be a symmetric function with the following properties:
\begin{align}\label{eq:sep}
\innp{u}{s_A}&>0\,;\\
\innp{u}{s_B}&\leq 0\quad\text{for all $B\neq A$ such that
$\phi(B)\in I$}\,.\notag
\end{align}

We show how to grow $I$ and $u$.  Let $\mu$ be a partition which
lies ``just below'' $I$, that is, $\mu\notin I$ and
$J=I\cup\{\mu\}$ is a dual order ideal.
Let 
$$
m=\max_{B\in \wssp_{\mu}}\{\innp{u}{s_B}\}\,.  
$$
If $m\leq 0$, then $u$ satisfies~\eqref{eq:sep} for $J$. 
Otherwise, let $u^*=u-ms_{\mu}$.
We show that $u^*$ satisfies~\eqref{eq:sep} for $J$.

For any $B$ such that $\phi(B)\in I$, we have
$\mu\ndom\phi(B)$, since $\mu\notin I$.  Therefore,
by Proposition~\ref{Pp:domprop},
$\innp{s_{\mu}}{s_B}=0$, and we have
$$
\innp{u^*}{s_B}=\innp{u}{s_B}
\begin{cases}
\leq0&\text{if $B\neq A$}\\
>0&\text{if $B=A$}
\end{cases}\,.
$$

For $B\in \wssp_{\mu}$, by Proposition~\ref{Pp:domprop},
$\innp{s_{\mu}}{s_B}=1$, so
we have
$$
\innp{u^*}{s_B}=\innp{u}{s_B}-m\leq0\,.
$$

The proof now proceeds by iterating this construction.
If $f$ separates $A$ from above, then
$f$ satisfies~\eqref{eq:sep} for the dual order
ideal generated by $\lambda$.  By iterating the contruction above,
we eventually arrive at a function $g$ which satisfies~\eqref{eq:sep}
for $I$ equal to the entire dominance poset.  This is the
same as saying $g$ separates $A$.
\end{proof}

An important special case is the following corollary.

\begin{Co}\label{Co:SSP1case}
If $|\wssp_{\lambda}|=1$, that is, $\wssp_{\lambda}=\{A\}$, then
$A$ is extreme in $\mathcal C_N^2$.
\end{Co}
\begin{proof}
The function $s_{\lambda}$ separates $A$ from above.
\end{proof}

We can limit our search for separating functions even further by restricting
to an interval in the dominance poset.  Suppose $\phi(A)=\lambda$
and $\rho\dom\lambda$.  We will say the symmetric function $f$
\textit{separates} $A$ on $[\lambda,\rho]$ if
\begin{enumerate}
\item
$f$ is Schur integral;
\item
$\innp{f}{s_{\mu}}=0$ whenever $\mu\notin[\lambda,\rho]$.
That is, the \textit{support of $f$} lies on $[\lambda,\rho]$;
\item
$\innp{f}{s_A}>0$;
\item
$\innp{f}{s_B}\leq 0$ for all $B$ such that
$\phi(B)\in[\lambda,\rho]$, $B\neq A$.
\end{enumerate}

\begin{Le}\label{Le:intervalseparate}
If $f$ separates $A$ on $[\lambda,\rho]$, then
$f$ separates $A$ from above.
\end{Le}

\begin{proof}
We show that for $B$ such that $\phi(B)\sdom\lambda$ but
$\rho\ndom\phi(B)$, we have $\innp{f}{s_B}=0$.
The support of $s_B$ is $\dom\phi(B)$ (Proposition~\ref{Pp:domprop}).
But then the support of $s_B$ cannot be below $\rho$, so
the support of $s_B$ does not intersect the interval
$[\lambda,\rho]$.
\end{proof}

Again suppose $\phi(A)=\lambda$, where $A\in\wssp_N$.
Two intervals above $\lambda$ in dominance will be of particular
interest to us.  First, if 
$\lambda=(\lambda_1\geq\lambda_2\geq\dots\geq\lambda_m>0)$, define
$$
\lambda^+=(\lambda_1+1,\lambda_2,\lambda_3,\dots,\lambda_{m-1},\lambda_m-1)\,.
$$
For example, if $\lambda=(4,3,3,2,2,2,1)$, then $\lambda^+=(5,3,3,2,2,2)$.
In the next section, we will find a symmetric function $f$ which separates $A$
on $[\lambda,\lambda^+]$ when $\lambda$ has distinct parts.  However,
this interval is not sufficient when $\lambda$ has repeated parts.  For example,
if $\lambda=2^31^3$, then no such $f$ separates
$A=\{(2,1),(2,1),(2,1)\}$ on this interval.

Now define
$$
\lambda^{++}=
\begin{cases}
(\lambda_1+1,\lambda_2+1,\dots,\lambda_k+1,
\lambda_{k+1}-1,\dots,\lambda_m-1)&\text{if $m=2k$}\\
(\lambda_1+1,\lambda_2+1,\dots,\lambda_k+1,\lambda_{k+1},
\lambda_{k+2}-1,\dots,\lambda_m-1)&\text{if $m=2k+1$}\,.
\end{cases}
$$
In the previous example,
$\lambda^{++}=(5,4,4,2,1,1)$.

\begin{Cj}\label{Cj:plusplusconj}
For every $A\in\wssp_N$ with $\phi(A)=\lambda$ there is a symmetric
function $f$ such that $f$ separates $A$ on $[\lambda,\lambda^{++}]$.
\end{Cj}

We have verified Conjecture~\ref{Cj:plusplusconj} for $N\leq20$.

\section{Distinct partitions}\label{S:pd}

In this section we show that if $A$ is nested and if $\phi(A)$ has distinct parts, then
$s_A$ is extreme.

\begin{Th}\label{Th:distinct}
If $\lambda$ has distinct parts and $A\in\wssp_{\lambda}$, then
there is a symmetric function $f$ which separates $A$ on
the interval $[\lambda,\lambda^+]$.
\end{Th}

Our strategy for proving Theorem~\ref{Th:distinct} is to find a chain of
subsets in $\wssp_{\lambda}$, starting with $\wssp_{\lambda}$ itself and ending
with $\{A\}$, such that there is a vector which separates the $(i+1)$ subset
in the chain from the $i$ subset. Putting these separating vectors together
produces a separating vector for $A$.

To make this strategy precise, we need some technical definitions and
lemmas.  Suppose $Y\subseteq\wssp_{\lambda}$.  We say symmetric function
$f$ \textit{separates
$Y$ on} $[\lambda,\mu]$ if
\begin{enumerate}
\item
$f$ is Schur-integral.
\item
$\innp{f}{s_B}=0$ for all $B\in\wssp_{\nu}$, $\nu\notin[\lambda,\mu]$.
\item
$\innp{f}{s_B}\leq0$ for all $B\in\wssp_{\nu}$, $\nu\in[\lambda,\mu]$,
$\nu\neq\lambda$.
\item
$\innp{f}{s_B}\leq0$ for all $B\in\wssp_{\lambda}-Y$.
\item
$\innp{f}{s_A}=k>0$ for all $A\in Y$, where integer $k$ does not depend on $A$.
\end{enumerate}

Now suppose $X\subseteq Y\subseteq\wssp_{\lambda}$.  We say 
symmetric function $g$ 
\textit{partially separates} $(X, Y)$ on $[\lambda,\mu]$ if
\begin{enumerate}
\item
$g$ is Schur-integral.
\item
$\innp{g}{s_B}=0$ for all $B\in\wssp_{\nu}$, $\nu\notin[\lambda,\mu]$.
\item
$\innp{g}{s_B}\leq0$ for all $B\in\wssp_{\nu}$, $\nu\in[\lambda,\mu]$,
$\nu\neq\lambda$.
\item
$\innp{g}{s_B}\leq0$ for all $B\in Y-X$.
\item
$\innp{g}{s_A}=l>0$ for all $A\in X$, where integer $l$ does not depend on $A$.
\end{enumerate}

Note that the sign of $\innp{g}{s_B}$ is not specified for $B\in\wssp_{\lambda}-Y$.

\begin{Le}\label{Le:seplemma}
Suppose $f$ separates $Y$ on $[\lambda,\mu]$ and $g$ partially separates
$(X, Y)$ on $[\lambda,\mu]$.  Then there exists an $h$ which separates
$X$ on $[\lambda,\mu]$.
\end{Le}
\begin{proof}
Let
$$
m=\max_{B\in\wssp_{\lambda}-Y}\innp{g}{s_B}\,.
$$
Pick a non-negative integer $b\geq m/k$.  Let
$$
h=g+bf-bks_{\lambda}\,.
$$
We now verify that $h$ has the required properties.  Clearly, $h$ is
Schur-integral and its support lies on $[\lambda,\mu]$.

Now suppose $B\in\wssp_{\nu}$, $\nu\in[\lambda,\mu]$, $\nu\neq\lambda$.
Then $\innp{g}{s_B}\leq 0$, $\innp{f}{s_B}\leq 0$, and $\innp{s_{\lambda}}{s_B}=0$
(by Proposition~\ref{Pp:domprop}).
Thus
$$
\innp{h}{s_B}\leq0\,,
$$
since $b\geq 0$.

Next, suppose $B\in\wssp_{\lambda}-Y$.
Then $\innp{g}{s_B}\leq m$, $\innp{f}{s_B}\leq 0$, and $\innp{s_{\lambda}}{s_B}=1$
(again by Proposition~\ref{Pp:domprop}).
Then
$$
\innp{h}{s_B}\leq m-bk\leq 0\,,
$$
since $b\geq 0$ and $b\geq m/k$.

Next, suppose $B\in Y-X$.  Then $\innp{g}{s_B}\leq 0$, $\innp{f}{s_B}=k$,
and $\innp{s_{\lambda}}{s_B}=1$.  Thus
$$
\innp{h}{s_B}\leq bk-bk=0\,.
$$

Finally, suppose $A\in X$.  Then
$\innp{g}{s_A}=l$, $\innp{f}{s_A}=k$, and $\innp{s_{\lambda}}{s_A}=1$, so
$$
\innp{h}{s_A}=l+bk-bk=l>0\,,
$$
and $l$ is independent of the choice of $A$.
\end{proof}

\begin{Co}\label{Co:chain}
If $\lambda=\phi(A)$ and 
$$
\wssp_{\lambda}=X_0\supseteq X_1\supseteq X_2\supseteq\dots\supseteq X_t=\{A\}
$$
and for each $i=0,1,\dots,t-1$ there is an $f_i$ which partially separates $(X_{i+1},X_{i})$
on $[\lambda,\mu]$, then there is a symmetric function $g$ which separates $A$
on $[\lambda,\mu]$.
\end{Co}

\begin{proof}
Clearly $s_{\lambda}$ separates all of $\wssp_{\lambda}$ on any interval. Iteratively
applying Lemma~\ref{Le:seplemma} in this chain of subsets yields $g$ which
separates $\{A\}$ on $[\lambda,\rho]$.  But that is the same as $g$ separates
$A$ on $[\lambda,\rho]$.
\end{proof}

Note that up to this point, we have not used the fact that $\lambda$ is distinct.
From now on, we assume $\lambda$ is distinct.

Suppose $A$, $B\in\wssp_{\lambda}$ and $\rho=(\lambda_i,\lambda_j)$,
but $\rho$ not necessarily
in either $A$ or $B$.
We say $A$ and $B$ \textit{agree
within} $\rho$ if
\begin{enumerate}
\item
if $(\lambda_u,\lambda_v)\in A$ (resp. $B$) and either $u$ or $v$ is
between $i$ and $j$, then
$i<u<v<j$;
\item
if $i<u<v<j$, then $(\lambda_u,\lambda_v)\in A$ if and only if
$(\lambda_u,\lambda_v)\in B$
\end{enumerate}

In addition, we say $A$ and $B$ \textit{agree on} $\rho$ if they agree within
$\rho$ and $\rho\in A, B$.

Note that these two definitions allow a part of size $1$ to be within $\rho$. 
That part must be in both $A$ and $B$.
Also define
$$
\lambda[\rho]=(\lambda_1,\lambda_2,\dots,\lambda_{i-1},
\lambda_i+1,\lambda_{i+1},
\dots,\lambda_{j-1},\lambda_j-1,\lambda_{j+1},\dots,\lambda_m)\,.
$$
Note that $\lambda[\rho]\in[\lambda,\lambda^+]$.

We illustrate these definitions with an example. Let
$$
\lambda=(17,16,15,13,12,11,9,8,7,5,4,2)\,,
$$
and
\begin{gather*}
A=(17,2),(16,7),(15,11),(13,12),(9,8),(5,4)\\
B=(17,5),(16,7),(15,11),(13,12),(9,8),(4,2)\\
C=(17,16),(15,11),(13,12),(9,8),(7,2),(5,4)
\end{gather*}
Let $\rho=(16,7)$. 
Then $A$, $B$ and $C$ all agree within $\rho$, and $A$ and $B$ agree on $\rho$.
Finally,
$$
\lambda[\rho]=(17,17,15,13,12,11,9,8,6,5,4,2)\,.
$$

The following lemma is crucial to our proof of Theorem~\ref{Th:distinct}.
\begin{Le}\label{Le:LRlemma}
Suppose $A$, $B\in\wssp_{\lambda}$, $\lambda$ distinct.
Suppose $\rho=(\lambda_i,\lambda_j)$, with
$\rho\in A$, $\rho\notin B$, and $A$ and $B$ agree within
$\rho$.
Then the Littlewood-Richardson coefficients satisfy the following identity:
$$
c_A^{\lambda[\rho]}+1 = c_B^{\lambda[\rho]}\,.
$$
Furthermore, if $j=i+1$, then
$c_A^{\lambda[\rho]}=0$
and
$c_B^{\lambda[\rho]}=1$.
\end{Le}

We defer the proof of Lemma~\ref{Le:LRlemma} for the moment and
show how it leads directly to a proof of Theorem~\ref{Th:distinct}.

\begin{proof}[Proof of Theorem~\ref{Th:distinct}.]
Suppose $A\in\wssp_{\lambda}$ where $\lambda=(\lambda_1,\dots,\lambda_n)$.
If $n=2m$ is even, then
all the partitions in $A$ are $2$-part partitions.  If $n=2m-1$
is odd, then exactly one partition in $A$ has one part.

Order all the partitions in $A$ with two parts from the ``inside out.''
That is, list the partitions in $A$ as $\rho^1,\rho^2,\dots$, where
all the partitions within $\rho^j$ appear among $\rho^1$, $\rho^2$,
$\dots$, $\rho^{j-1}$. If $\lambda$ is odd, put the $1$-part partition last in the above list.

Let
$$
X_i=\{B\in\wssp_{\lambda}\mid\text{$B$ and $A$ agree on $\rho^1$,
$\rho^2$, $\dots$, $\rho^i$}\}\,.
$$
We then have this chain of subsets:
$$
\wssp_{\lambda}=X_0\supseteq X_1\supseteq\dots\supseteq X_m=\{A\}\,.
$$
We wish to apply Corollary~\ref{Co:chain} to this chain, so we need to construct
$f_i$ which partially separates $(X_{i+1},X_i)$ on $[\lambda,\lambda^+]$.
Note that $X_i$ consists of all elements of $\wssp_{\lambda}$ which agree
with $A$ on $\{\rho^1,\dots,\rho^i\}$ and $X_i-X_{i+1}$ are those which
do not contain $\rho = \rho^{i+1}$.  That is, $B\in X_i$ means $B$
and $A$ agree within $\rho$, but for such $B$, $B\in X_{i+1}$ if and
only if $\rho\in B$.  These are exactly the conditions needed to apply
Lemma~\ref{Le:LRlemma}.

Let
$$
f_i=(c_A^{\lambda[\rho]}+1)s_{\lambda}-s_{\lambda[\rho]}\,.
$$
We now verify that $f_i$ partially separates $(X_{i+1}, X_i)$ on $[\lambda,\lambda^+]$,
thus completing the proof.

Clearly $f_i$ is Schur-integral and its support lies on $[\lambda,\lambda^+]$.

For $B\in\wssp_{\nu}$ with $\nu\in[\lambda,\lambda^+]$ and $\nu\neq\lambda$, we have $\innp{s_{\lambda}}{s_B}=0$
(by Proposition~\ref{Pp:domprop})
and $\innp{s_{\lambda[\rho]}}{s_B}\geq 0$ (since $s_B$ is Schur-positive).
Thus $\innp{f_i}{s_B}\leq 0$.

For $B\in Y-X$, $\innp{s_{\lambda}}{s_B}=1$ (by Proposition~\ref{Pp:domprop})
and $\innp{s_{\lambda[\rho]}}{s_B}=c_B^{\lambda[\rho]}$.  So by
Lemma~\ref{Le:LRlemma},
$\innp{f_i}{s_B}=c_A^{\lambda[\rho]}+1-c_B^{\lambda[\rho]}=0$.

For $B\in X$, $\innp{s_{\lambda}}{s_B}=1$ (by Proposition~\ref{Pp:domprop})
and $\innp{s_{\lambda[\rho]}}{s_B}=c_A^{\lambda[\rho]}$.  Then
$\innp{f_i}{s_B}=c_A^{\lambda[\rho]}+1-c_A^{\lambda[\rho]}=1$.

Therefore $f_i$ partially separates $(X_{i+1}, X_i)$ as required.

\end{proof}

We make a couple of observations about $\wssp_{\lambda}$ when $\lambda$
is distinct.  
If $\lambda$ is even ($n=2m$), then the elements of $\wssp_{\lambda}$ are
clearly counted by the Catalan numbers $C_m$.  If $\lambda$ is odd
($n=2m-1$) then the elements of
$\wssp_{\lambda}$ are
counted again by the Catalan numbers $C_m$.
Furthermore, for our ordering of the $\rho^i$'s,
we may use the natural Catalan recursion induced by our
realization of the partitions in $\wssp_{\lambda}$ as nestings.
See Exercise 6.19, part (o) in~\cite{Stan2}.

The awkward distinction between even $n$ and odd $n$ can be resolved
in several ways.  Our proof above placed the singleton part at the end
so that all the work had been accomplished before encountering it.  
Since a
$1$-part is incorporated in Lemma~\ref{Le:LRlemma}, this was not technically
necessary.  A
heuristic of the general (non-distinct) problem seems to be that the odd case
is  easier than the even case.

\section{A Littlewood-Richardson identity}\label{S:LRIdentity}

Our goal in this section is to prove
Lemma~\ref{Le:LRlemma}. This lemma will be a corollary of a stronger
theorem which we now describe.

We generalize somewhat the notation from Section~\ref{S:extreme}.
Suppose $A$ is a multiset in $\mathcal P_N$ (not necessarily $1$ or $2$
row partitions). As with $2$-part partitions, let
$\phi(A)$ be the partition defined by the parts of the
partitions in $A$.  
Let $n=l(\phi(A))$, the
number of parts of $\phi(A)$.

If $\phi(A)$ is distinct, the location in $\phi(A)$ of each part of
each partition in $A$ defines a set partition of $\{1,\dots,n\}$ with $m$ blocks,
where $m$ is the number of partitions in $A$.
If $\rho^i$ is a partition in $A$, let $\alpha^i$ denote the corresponding block of
the set partition.

For example, if $\rho^1=(8,3,1)$ and $\rho^2=(6,4)$, and $A=\{\rho^1,\rho^2\}$
(with $n=5$, $m=2$ and $N=22$), then $\phi(A)=(8,6,4,3,1)$,
and $\alpha^1=\{1,4,5\}$ and $\alpha^2=\{2,3\}$.

We will be forming tableaux
with content equal to various rearrangements of $\phi(A)$.  Since the elements of the 
blocks $\alpha^i$ will correspond to letters in such tableaux, we view these
blocks as alphabets.
When switching order between alphabets,
we will keep the letters within alphabets intact, rather than relabel. 
Also, note that since $\phi(A)$ is distinct, there is no ambiguity in the definition of 
the set partition $\{\alpha^1,\dots,\alpha^m\}$.

Now for $\lambda\partof N$ define $d_A^{\lambda}$ to be the number of SSYT 
$T$ of shape
$\lambda$, content $\phi(A)$, such that each word $w_{\alpha^i}(T)$ is LW.

Here is an example of such a SSYT for the above $A$ and $\lambda=(9,9,3,1)$:

$$
\begin{array}{ccccccccc}
\clr1&\clr1&\clr1&\clr1&\clr1&\clr1&\clr1&\clr1&\clb2\\
\clb2&\clb2&\clb2&\clb2&\clb2&\clb3&\clr4&\clr4&\clr4\\
\clb3&\clb3&\clb3\\
\clr5
\end{array}
$$

The words
$$
w_{\alpha^1}(T)=(\clr{1\,1\,1\,1\,1\,1\,1\,1\,4\,4\,4\,5})
$$
and
$$
w_{\alpha^2}(T)=(\clb{2\,3\,2\,2\,2\,2\,2\,3\,3\,3})
$$
are both lattice words.

The calculation of $d_A^{\lambda}$ is similar to the Littlewood-Richardson calculation, except
the content has been sorted into decreasing order
and the subwords corresponding to each $\rho\in A$ are scattered throughout the tableau.

In general $d_A^{\lambda}\neq c_A^{\lambda}$. We shall give examples of this inequality later 
in this section.
However, in one particular case, these two numbers are equal, and their equality allows us to calculate
$c_A^{\lambda}$ exactly.

Specifically, suppose $\lambda\partof N$, $\lambda$ with distinct parts, and $\nu=(\lambda_i,\lambda_j)$
is a two-part partition using two parts of $\lambda$. We consider SSYT of shape $\lambda[\nu]$ and content
$\lambda$.

We first characterize these tableaux.  
Let $T$ be such a tableau.  Suppose $u<i$ or
$u\geq j$. Then all the entries in row $u$ of $T$ are $u$'s.  Furthermore, if $i\leq u<j$, then all the entries in row
$u$ of $T$ are $u$'s, except for possibly the last one. We call the last cell in each of these rows
\textit{special}.  A special cell does not have a cell
below it, since the only possible repeated rows in $\lambda[\nu]$ are rows $i-1$ and $i$ or rows $j$ and $j+1$.

The entries in the special cells (reading from row $j-1$ up to
row $i$) will form a permutation of $\{(i+1),\dots,j\}$.  However, for purposes which will soon become clear,
we write this as a permutation of $\{i,\dots,j\}$ by inserting an $i$ in the next-to-last position.
We call this permutation $\pi=\pi_1,\dots,\pi_{j-i+1}$. Note that since $\pi_1$ is the entry in row $j-1$,
$\pi_1$ can only be $j$ or $j-1$.  Similarly, $\pi_2$ can only be $j-2$ or the value not used for $\pi_1$.
Continuing in this fashion,
there are only two choices for each entry in $\pi$, except the last two.
This description completely characterizes these SSYT.

\begin{Pp}\label{Pp:SSYTcount}
The number of SSYT of shape $\lambda[\nu]$ and content $\lambda$ is $2^{j-i-1}$ (that is, the Kostka number,
$K_{\lambda[\nu],\lambda}=2^{j-i-1}$).
\end{Pp}

Suppose $i\leq u<v\leq j$.  If $u\neq i$, it follows from the discussion above that $w_{\{u,v\}}(T)$ will be LW if and only if
$u$ follows $v$ in $\pi$.  Our insertion of $i$ into $\pi$ extends this to the $u=i$ case.

For example, let $\lambda=(7,5,3,2,1)$ and $\nu=(5,1)$.  There are $4$ SSYT of shape $\lambda[\nu]$
and content $\lambda$.  We list them with their corresponding $\pi$:
$$
A_1=
\begin{array}{ccccccc}
1&1&1&1&1&1&1\\
2&2&2&2&\clr2&\clr5\\
3&3&\clr3\\
4&\clr4\\
\end{array}
\pi=\clr{4\,3\,2\,5}
$$

$$
A_2=
\begin{array}{ccccccc}
1&1&1&1&1&1&1\\
2&2&2&2&\clr2&\clr3\\
3&3&\clr5\\
4&\clr4\\
\end{array}
\pi=\clr{4\,5\,2\,3}
$$

$$
A_3=
\begin{array}{ccccccc}
1&1&1&1&1&1&1\\
2&2&2&2&\clr2&\clr4\\
3&3&\clr3\\
4&\clr5\\
\end{array}
\pi=\clr{5\,3\,2\,4}
$$

$$
A_4=
\begin{array}{ccccccc}
1&1&1&1&1&1&1\\
2&2&2&2&\clr2&\clr3\\
3&3&\clr4\\
4&\clr5\\
\end{array}
\pi=\clr{5\,4\,2\,3}
$$
From this example it is easy to verify, for instance, that $w_{\{2,3\}}(A_1)$ and
$w_{\{2,3\}}(A_3)$ have the LW property but
$w_{\{2,3\}}(A_2)$ and $w_{\{2,3\}}(A_4)$ do not.
Also $w_{\{3,4\}}(A_1)$,  $w_{\{3,4\}}(A_2)$ and
$w_{\{3,4\}}(A_4)$ have the LW property, but 
$w_{\{3,4\}}(A_3)$ does not.  This is reflected in the corresponding $\pi$'s.

\begin{Th}\label{Th:LRCount}
Suppose $\lambda\partof N$ is distinct, $\nu=(\lambda_i,\lambda_j)$, and $A$ is a multiset
of partitions, with
$\phi(A)=\lambda$. Then
$$
c_A^{\lambda[\nu]}=d_A^{\lambda[\nu]}\,.
$$
Furthermore, $c_A^{\lambda[\nu]}$ is the number of permutations $\pi=\pi_1,\dots,\pi_{j-i+1}$
of $\{i,\dots,j\}$ such that
\begin{enumerate}
\item
\begin{align*}
&\pi_t\geq j-t\text{ for }t=1,2,\dots,j-i-1\\
&\pi_{j-i}=i\\
&\pi_{j-i+1}\text{ is the only remaining value.}
\end{align*}
\item
if $u<v$ are in $\alpha^i$, then
$v$ appears before $u$ in $\pi$.
\end{enumerate}
\end{Th}
\begin{proof}
The characterization of the permutations follows from the discussion above.

Our strategy in proving the equality is to swap alphabets using Theorem~\ref{Th:swapprop}.
Starting with a SSYT $T$ counted by $d_A^{\lambda[\nu]}$, as noted above, we 
view the blocks of the
set partition $\{\alpha^1,\dots,\alpha^m\}$ as alphabets.  Each alphabet
is scattered throughout the tableau, not contiguously, so Theorem~\ref{Th:swapprop} does not
directly apply.  To circumvent this, we use Theorem~\ref{Th:swapprop} on a portion of the tableau
where the alphabets do appear contiguously.

Suppose we have constructed a new tableau $S$ from $T$, with special row $k$.
Suppose $S$ and $k$ have the following properties.
\begin{enumerate}
\item
The shape of $S$ is $\lambda[\nu]$.
\item
The content of $S$ is a rearrangement of $\lambda$. As noted above, we keep the same
alphabets as $T$, but reordered without relabeling.
\item
Each subword corresponding to each alphabet $\alpha^u$ is a lattice word.
\item
At and above row $k$, $S$ is identical to $T$. We shall say $S$ is
\textit{pristine above} $k$.
\item
The letters $>k$ in each alphabet $\alpha^u$ appear contiguously in $S$.
We shall say $S$ is \textit{clustered below} $k$. Note that ``below'' here
means ``greater than''.  Letters $>k$ can appear in the pristine portion of $S$.
\end{enumerate}

We will call such $S$ $k$-\textit{partially LR}.

Let $n=l(\lambda[\rho])$. (Note that $n$ may or may not equal $l(\lambda)$, depending on
whether $\nu$ includes the last part of $\lambda$ and that part is $1$.)
Initially, $T$ is $n$-partially LR.  Also note that if $S$ is $0$-partially LR,
then it is counted by $c_A^{\lambda[\nu]}$.

Now suppose
$S$ is $k$-partially LR. We show how to construct a new tableau
$S'$ which is $(k-1)$-partially LR.  Iterating  will  give
 a $0$-partially LR tableau.  Each step in this process will be seen to be
reversible, thus proving the result.

Suppose $k$ is in block $\alpha^{u_0}$.  If $k$ is the largest in its block, then $S$ will be clustered
below $k-1$ and obviously pristine above $k-1$, so let $S'=S$.

Now suppose $k$ is not the largest in its block. Since $S$ is pristine above $k$,
except possibly at the special cells,
no letter greater than $k$ appears above row $k$. Any swaps that take place between letters
greater than $k$ will not affect the special cells, since there is nothing below them,
and so the tableau above row $k$ will remain
pristine.

Define for each $u$, $\beta^u=\alpha^u\cap\{k+1,\dots,n\}$.
Since $k$ is not largest in its block, $\beta^{u_0}$ is non-empty.  Since $S$ is clustered
below $k$, define $S^u$ to be the skew subtableau of $S$ containing the letters
in $\beta^u$.  We then apply Theorem~\ref{Th:swapprop}
to move $S^{u_0}$ through the $S^u$ until the smallest
letter in $\beta^{u_0}$ and $k$ are contiguous.
This new tableau is $S'$.
As noted above, these moves will leave the rows above $k$
unaffected, so that $S'$ is pristine above $k$.
Furthermore, each of these moves will preserve the LW property
within each $S^u$ (by Theorem~\ref{Th:swapprop}), even if
some of the letters appear in special cells in the pristine region.

It remains to verify that the LW property holds in $S'$ for each $\alpha^u$.
We say $x, y\in\alpha^u$ are \textit{sequential} if there is no
$z\in\alpha^u$ between $x$ and $y$. Note that this is different than
contiguous:  contiguous means there is no $z$ between
two entries in the tableau, while sequential means there is no $z$
between two entries in the alphabet.  Two letters could be
sequential  (same alphabet) but not contiguous (intervening letters from
different alphabets).  And two letters could be contiguous (no intervening
letter within the tableau) but not sequential (from different alphabets).

It suffices to show the LR
property holds in $S'$ for all sequential pairs in $\alpha^u$.

We consider three possible cases for sequential pair $x,y$.  
First, suppose $x,y>k$.  Then $x$ and $y$ will have the LW
property in $\alpha^u$ in $S$ if and only if
they have the LW property in $S'$ by Theorem~\ref{Th:swapprop}.

Second, suppose $x<k$ and $y\leq k$.
Then neither $x$ nor $y$ is in $\beta^u$, so their relative positions will be
unchanged by swapping alphabets.

Third, suppose $x\leq k$ and $y>k$.
All the $x$'s appear in the pristine portion of $S$ and so remain unchanged.
All the $y$'s (except for possibly one in a special cell) are below row $k$.
Therefore if $x$ and $y$ had the LW property before a swap, it would have it
afterwards, and conversely.
\end{proof}
We illustrate with some examples.  First, suppose 
\begin{gather*}
\alpha^1=\{\clr4,\clr8,\clr9,\clr{13},\clr{14},\clr{16}\}\\
\alpha^2=\{\clb2,\clb5,\clb6,\clb7,\clb{11},\clb{15}\}\\
\alpha^3=\{\clg1,\clg3,\clg{10},\clg{12}\}\,.
\end{gather*}
Suppose $S$ is $\clb7$-partially LR.
Then
\begin{gather*}
\beta^1=\{\clr8,\clr9,\clr{13},\clr{14},\clr{16}\}\\
\beta^2=\{\clb{11},\clb{15}\}\\
\beta^3=\{\clg{10},\clg{12}\}\,.
\end{gather*}

Furthermore, suppose
$\beta^1<\beta^3<\beta^2$.  That is, the order on the entries in $S$ due to previous switches is
$$
\clr8<\clr9<\clr{13}<\clr{14}<\clr{16}<\clg{10}<\clg{12}<\clb{11}<\clb{15}\,.
$$
Since $\clb7$ belongs to
$\alpha^2$, $\beta^2$ will switch with $\beta^3$ then $\beta^1$.  The new order will be
$$
\clb7<\clb{11}<\clb{15}<\clr8<\clr9<\clr{13}<\clr{14}<\clr{16}<\clg{10}<\clg{12}\,.
$$

In this new order, $\clb{15}$ and $\clr8$ are contiguous but not sequential (different alphabets).
And $\clr4$ and $\clr8$ are sequential but not contiguous (separated by
$\clb5, \clb6, \clb7, \clb{11},\clb{15}$).

Our second example illustrates the switching within a tableau.
Let
$$
\lambda=(9,8,7,6,5,4,3,2,1)
$$
and $\mu=(9,1)$. Then 
$$
\lambda[\mu]=(10,8,7,6,5,4,3,2)\,.
$$
Let $A=\{\rho^1,\rho^2,\rho^3\}$, with $\rho^1=(9,5,3)$, $\rho^2=(8,7,4,2)$ and $\rho^3=(6,1)$.  Note that 
$\phi(A)=\lambda$.
Also, since $9$, $5$ and $3$ are in the $1$, $5$ and $7$ positions of $\lambda$,
$\alpha^1=\{\clr1,\clr5,\clr7\}$.  Similarly, $\alpha^2=\{\clb2,\clb3,\clb6,\clb8\}$ and $\alpha^3=\{\clg4,\clg9\}$.
Initially let
$$
T=
\begin{array}{cccccccccc}
\clr1&\clr1&\clr1&\clr1&\clr1&\clr1&\clr1&\clr1&\clr1&\clg4\\
\clb2&\clb2&\clb2&\clb2&\clb2&\clb2&\clb2&\clb2\\
\clb3&\clb3&\clb3&\clb3&\clb3&\clb3&\clb3\\
\clg4&\clg4&\clg4&\clg4&\clg4&\clr5\\
\clr5&\clr5&\clr5&\clr5&\clg9\\
\clb6&\clb6&\clb6&\clb6\\
\clr7&\clr7&\clr7\\
\clb8&\clb8
\end{array}
$$
Note that $i=1$, $j=9$, and $\pi=8\,7\,6\,9\,5\,3\,2\,1\,4$.
Now suppose we have constructed the corresponding $3$-partially LR tableau $S$.
In fact, suppose
$$
S=
\begin{array}{cccccccccc}
\clr1&\clr1&\clr1&\clr1&\clr1&\clr1&\clr1&\clr1&\clr1&\clg4\\
\clb2&\clb2&\clb2&\clb2&\clb2&\clb2&\clb2&\clb2\\
\clb3&\clb3&\clb3&\clb3&\clb3&\clb3&\clb3\\
\hline
\clg4&\clg4&\clg4&\clg4&\clg4&\clr5\\
\clg9&\clr5&\clr5&\clr5&\clr5\\
\clr7&\clr7&\clr7&\clb6\\
\clb6&\clb6&\clb6\\
\clb8&\clb8
\end{array}
$$

Then $\beta^1=\{\clr5,\clr7\}$, $\beta^2=\{\clb6,\clb8\}$ and $\beta^3=\{\clg4,\clg9\}$. Note that
at this stage, $\beta^3<\beta^1<\beta^2$. 
Since $\clb3$ belongs to $\alpha^2$, $S^2$ will swap with $S^1$, then $S^3$, to form
$$
S'=
\begin{array}{cccccccccc}
\clr1&\clr1&\clr1&\clr1&\clr1&\clr1&\clr1&\clr1&\clr1&\clg4\\
\clb2&\clb2&\clb2&\clb2&\clb2&\clb2&\clb2&\clb2\\
\clb3&\clb3&\clb3&\clb3&\clb3&\clb3&\clb3\\
\hline
\clb6&\clb6&\clb6&\clb6&\clr5&\clr5\\
\clb8&\clb8&\clr5&\clr5&\clg4\\
\clr5&\clr7&\clr7&\clg4\\
\clr7&\clg4&\clg4\\
\clg4&\clg9
\end{array}
$$

Note that $T$, $S$ and $S'$ all have the appropriate LW property.

The new tableau $S'$ will be pristine above $k-1$ and clustered below $k-1$.

Lemma~\ref{Le:LRlemma} now follows as a corollary
\begin{proof}[Proof of Lemma~\ref{Le:LRlemma}]
By Theorem~\ref{Th:LRCount}, $c_A^{\lambda[\rho]}$ (resp. $c_B^{\lambda[\rho]}$)
counts permutations $\pi$
of $\{i,\dots,j\}$ such that $\pi_i\geq j-i$ for $i=1,2,\dots,j-i-1$ and $\pi_{j-i}=i$,
and if $(\lambda_r,\lambda_s)\in A$ (resp. $\in B$), then
$s$ appears before $r$ in $\pi$.  It is clear that exactly one such permutation 
has $i$ appearing before $j$, namely $j-1,j-2,\dots,i+1,i,j$. This permutation
is counted in $c_B^{\lambda[\rho]}$ but not in $c_A^{\lambda[\rho]}$. 

If $j=i+1$, then there is exactly one SSYT, $T$, and $\pi=(i,i+1)$, so this $T$ is LW
for $B$ but not for $A$.
\end{proof}

We conclude with two examples which illustrate how special is the case of Theorem~\ref{Th:LRCount}.
First, to demonstrate that the content must be in decreasing order,
let $\mu=(4,4,1,1)$ with $A=\{(4,1),(3,2)\}$.  Then
$\lambda=\phi(A)=(4,3,2,1)$.  Note that $\mu=\lambda[(3,2)]$.  From Theorem~\ref{Th:LRCount},
we know $c_A^{\mu}=d_A^{\mu}=0$ and $K_{\mu,\lambda}=1$. 
However, if we use a different order, $(3,4,1,2)$, for the content, we have
$$
T=
\begin{array}{cccc}
\clr1&\clr1&\clr1&\clb2\\
\clb2&\clb2&\clb2&\clr4\\
\clb3\\
\clr4
\end{array}
$$
The corresponding words in this tableau are LW.

The second example illustrates how important it is that the shape be only slightly different than the content.
Take
$\lambda=(6,5,4,3,2,1)$, $A=\{(6,1),(5,4),(3,2)\}$ and $\mu=(7,7,5,2)$.
Note that $\lambda$ is distinct and $A$ is nested, but $\mu$ is not $\lambda[\rho]$ for any
possible $\rho$.
Now let
$$
T=
\begin{array}{ccccccc}
\clr1&\clr1&\clr1&\clr1&\clr1&\clr1&\clg4\\
\clb2&\clb2&\clb2&\clb2&\clb2&\clg4&\clr6\\
\clb3&\clb3&\clb3&\clb3&\clg4\\
\clg5&\clg5\\
\end{array}
$$
Clearly $T$ has content $\lambda$ and the words corresponding to $A$ are LW.
However, repeated switches yields this tableau (with content
$(6,1,5,4,3,2)$):
$$
S=
\begin{array}{ccccccc}
\clr1&\clr1&\clr1&\clr1&\clr1&\clr1&\clr6\\
\clb2&\clb2&\clb2&\clb2&\clb2&\clg4&\clg4\\
\clb3&\clb3&\clb3&\clb3&\clg4\\
\clg5&\clg5\\
\end{array}
$$
Note that $S$ is not LW. In fact, $d_A^{\mu}=15$ while
$c_A^{\mu}=13$.

\section{Nestings and plane partitions}\label{S:planepart}

There is an interesting connection between the nested sets
of partitions described in Conjecture~
\ref{Cj:firstconj} and
and plane partitions.
See~\cite{Stan1} for the relevant definitions associated with plane partitions.

\begin{Pp}\label{Pp:planepartition}
If $\lambda\vdash N$ has $2m$ parts, then there is a one-to-one correspondence
$\psi$ between elements of $\wssp_{\lambda}$ and
plane partitions of $N$ with shape
$(m,m)$ and parts $\lambda_i$.
\end{Pp}
\begin{proof}
For each $\rho$ in $A$, place $\rho_1$ in the first row of $\psi(A)$
and $\rho_2$ in
the second row.  Then write the rows in decreasing order.
\end{proof}

We illustrate this bijection with the following table, when $\lambda=12^334^25$.
There are four elements of $\wssp_{\lambda}$.
$$
\begin{array}{|c|c|c|}
\hline
A\in\wssp_{\lambda}&\psi(A)\\
\hline
\{(5,1),(4,2),(4,2),(3,2)\}&
\begin{matrix}
5&4&4&3\\
2&2&2&1
\end{matrix}\\
\hline
\{(5,1),(4,2),(4,3),(2,2)\}&
\begin{matrix}
5&4&4&2\\
3&2&2&1
\end{matrix}\\
\hline
\{(5,1),(4,4),(3,2),(2,2)\}&
\begin{matrix}
5&4&3&2\\
4&2&2&1
\end{matrix}\\
\hline
\{(5,3),(4,4),(2,1),(2,2)\}&
\begin{matrix}
5&4&2&2\\
4&3&2&1
\end{matrix}\\
\hline
\end{array}
$$

When the number of parts of $\lambda$ is odd, the bijection
$\psi$ becomes an injection.
For example, suppose $\lambda=1^32^334^2$.
$$
\begin{array}{|c|c|c|}
\hline
A\in\wssp_{\lambda}&\psi(A)\\
\hline
\{(4,2),(4,2),(3,2),(1),(1,1)\}&
\begin{matrix}
4&4&3&1&1\\
2&2&2&1&0
\end{matrix}\\
\hline
\{(4,2),(4,3),(2,2),(1),(1,1)\}&
\begin{matrix}
4&4&2&1&1\\
3&2&2&1&0
\end{matrix}\\
\hline
\{(4,4),(3,2),(2,2),(1),(1,1)\}&
\begin{matrix}
4&3&2&1&1\\
4&2&2&1&0
\end{matrix}\\
\hline
\{(4,4),(3),(2,1),(2,1),(2,1)\}&
\begin{matrix}
4&3&2&2&2\\
4&1&1&1&0
\end{matrix}\\
\hline
\{(4,4),(3),(2,1),(2,2),(1,1)\}&
\begin{matrix}
4&3&2&2&1\\
4&2&1&1&0
\end{matrix}\\
\hline
\end{array}
$$
To illustrate that the mapping to plane partitions is not
a bijection, the plane partition
$$
\begin{matrix}
4&4&3&2&1\\
2&2&1&1&0
\end{matrix}
$$
does not appear in this list.

\section{Remarks and acknowledgements}\label{S:remarks}

It is easy to describe the extreme vectors of
$\mathcal C_N^{N-1}$.  These vectors are all the Schur functions
except $s_{1^n}$, which is replaced by $s_{1^{n-1}}s_1$.
In a similar (but more complicated) fashion, it is possible
to describe the extreme vectors of $\mathcal C_N^{N-2}$
and $\mathcal C_N^{N-3}$.

However, we have been unable to replace the conditions in
Conjecture~\ref{Cj:firstconj} with general conditions for
the cone $\mathcal C_N^k$.  Even the case $k=3$
seems difficult, requiring that the syzygies
in Equations~\eqref{eq:syzygy1} to~\eqref{eq:syzygy4}
be replaced with appropriate syzygies for 3-row partitions.

Lemma~\ref{Co:chain} gives a general recipe for constructing
separating vectors.  Unfortunately, if $\lambda$ has repeated
parts, the chain of subsets of $\wssp_{\lambda}$
and the required partially separating
vectors seem much more elusive than in the case of distinct parts.

The author would like to thank Alexander Yong for suggesting
placing the log concave problem into the context of symmetric
functions.

\end{document}